\documentclass[12pt,a4paper]{article}

\usepackage[backref=false,style=alphabetic]{biblatex}
\bibliography{dynsys}

\usepackage[left=3cm,right=3cm,top=3.5cm,bottom=3.5cm]{geometry} % we can adjust these values (making sure that title and TOC page look good)
\usepackage[utf8]{inputenc}
\usepackage[english]{babel}
\usepackage{amssymb}
\usepackage{amsmath}
\usepackage{latexsym}
\usepackage[bookmarks]{hyperref}
\usepackage{amsthm}
\usepackage{graphicx}
\usepackage{bbm}
\usepackage{verbatim}
\usepackage{epigraph}
\usepackage{mathtools}
\usepackage{authblk}
\usepackage[pdflatex]{crop}
\usepackage{tikz}
\usepackage{tikz-cd}
\usetikzlibrary{decorations.pathreplacing}
\usepackage{wasysym}
\usepackage{adjustbox}
\usepackage{microtype}
\usepackage{chngcntr}
\usepackage{csquotes}
\usepackage{enumerate}
\usepackage[capitalize]{cleveref}
\usepackage{stmaryrd}	% provides \ovee
\allowdisplaybreaks
\usetikzlibrary{shapes.geometric,fit}
\usepackage{faktor}

\usepackage{epigraph}
\numberwithin{equation}{section}

\theoremstyle{plain}
\newtheorem{thm}{Theorem}
\newtheorem{lemma}[thm]{Lemma}
\newtheorem{prop}[thm]{Proposition}
\newtheorem{cor}[thm]{Corollary}

\newtheorem{deph}[thm]{Definition}

\theoremstyle{definition}
\newtheorem{remark}[thm]{Remark}
\newtheorem{eg}[thm]{Example}

\newcommand{\Z}{\mathbb{Z}}
\newcommand{\N}{\mathbb{N}}

\newcommand{\cat}[1]{{\mathsf{#1}}} % font for categories

\newcommand{\id}{\mathrm{id}} % identity
\newcommand{\join}{\vee}
\DeclareMathOperator*{\bigjoin}{\bigvee}

\DeclareMathOperator{\ent}{Ent}
\DeclareMathOperator{\qent}{qEnt}

\DeclareMathOperator{\mat}{Mat}

\mathchardef\hy="2D

\let\originalleft\left
\let\originalright\right
\renewcommand{\left}{\mathopen{}\mathclose\bgroup\originalleft}
\renewcommand{\right}{\aftergroup\egroup\originalright}

\title{Partially observable systems and quotient entropy via graphs}
\author[1]{Leonhard Horstmeyer\footnote{Correspondence: horstmeyer [at] csh.ac.at}\footnote{LH has been supported by the Austrian Science Fund (FWF) via the grant P29252.}}
\affil[1]{\small Complexity Science Hub Vienna, Austria}
\author[2]{Sharwin Rezagholi\footnote{Correspondence: sharwin.rezagholi [at] mis.mpg.de}}
\affil[2]{\small Max Planck Institute for Mathematics in the Sciences, Leipzig, Germany}
\date{September 2019}

\begin{document}

\maketitle

\begin{abstract}
\noindent We consider the category of partially observable dynamical systems, to which the entropy theory of dynamical systems extends functorially as quotient-topological entropy. We discuss the structure that emerges. We show how quotient entropy can be explicitly computed by symbolic coding. To do so, we make use of the relationship between the category of dynamical systems and the category of graphs, a connection mediated by Markov partitions and topological Markov chains.

\end{abstract}

\section{Introduction}

Often a dynamical system may only be partially observed. The most important invariant of dynamical systems, topological entropy, can be extended to partially observable dynamical systems as quotient-topological entropy, which quantifies the complexity reduction due to partial observability. Just as topological entropy, quotient-topological entropy is hard to compute explicitly. For topological Markov chains, and in turn certain partially observable dynamical systems with Markov partitions, we provide a scheme to calculate quotient-topological entropy. This scheme is the computation of the entropy of a topological Markov chain whose entropy is equal to the respective quotient-entropy. These chains are generated by images of the graphs that generate the symbolic representations of the base systems.

\section{Dynamical systems and partial observability}

The basic mathematical entities of this chapter are from the category $\cat{CHausSur}$ of \emph{compact Hausdorff topological spaces and continuous surjective maps}. A \emph{dynamical system} will be a tuple $(X,f)$ where $X$ is a topological space with a compact Hausdorff topology and $f: X \to X$ is a continuous surjection. It generates the continuous monoid action $X \times \N_0 \to X$ given by $(x,t) \mapsto f^t (x)$. These are the objects of the category $\cat{Sys}$. A morphism $m: (X,f) \to (Y,g)$ is a continuous surjection $m: X \to Y$ such that the following diagram commutes.
\begin{equation*}
\begin{tikzcd}
X \arrow{r}{f} \arrow{d}{m} & X \arrow{d}{m} \\
Y \arrow{r}{g} & Y
\end{tikzcd}
\end{equation*}

We define the category $\cat{PSys}$ of \emph{partially observable systems}. The objects of $\cat{PSys}$ are triples $(X,f,q_X)$ where $(X,f)$ is a dynamical system and $q: X \to \tilde X$ is a continuous surjection. A morphism from $(X,f,q_X)$ to $(Y,g,q_Y)$ is a tuple $(m,n)$, where $m:X \to Y$ and $n : \tilde X \to \tilde Y$ are continuous surjections, such that the following diagram commutes.
\begin{equation*}
\begin{tikzcd}
X \arrow{r}{f} \arrow{d}{m} & X \arrow{r}{q_X} \arrow{d}{m} & \tilde X \arrow{d}{n} \\
Y \arrow{r}{g} & Y \arrow{r}{q_Y} & \tilde Y \\
\end{tikzcd}
\end{equation*}

\begin{remark}
There is a forgetful functor $F : \cat{Sys} \to \cat{CHausSur}$ given by the assignments $(X,f) \mapsto X$ and $m \mapsto m$ on morphisms. The functor $F$ forgets the dynamics.
\begin{equation*}
\begin{tikzcd}
X \arrow{r}{f} \arrow{d}{m} & X \arrow{d}{m} \\
Y \arrow{r}{g} & Y
\end{tikzcd}
\, \xmapsto{F} \,
\begin{tikzcd}
X \arrow{d}{m} \\
Y
\end{tikzcd}
\end{equation*}
The pair of functors $\left( F, \id_{\cat{CHausSur}} \right)$ generates the comma category $\faktor{F}{\id}$. Its objects are triples of the form $\big( (X,f) , \tilde X , q_X \big)$ where $(X,f)$ is a dynamical system, $\tilde X$ is a compact Hausdorff space, and $q_X : X \to \tilde X$ is a continuous surjection. A morphism between $\big( (X,f) , \tilde X , q_X \big)$ and $\big( (Y,g) , \tilde Y , q_Y \big)$ is a tuple $(m,n)$, where $m: (X,f) \to (Y,g)$ is a morphism in $\cat{Sys}$ and $n: \tilde X \to \tilde Y$ is a morphism in $\cat{CHausSur}$. We obtain the category $\cat{PSys}$ by switching to compact notation.
\end{remark}

There is a forgetful functor $U : \cat{PSys} \to \cat{Sys}$ which forgets the partial observability. On objects we have $(X,f,q) \mapsto (X,f)$ and on morphisms we have $(m,n) \mapsto m$. There is a functorial embedding $E: \cat{Sys} \hookrightarrow \cat{PSys}$ where on objects we have $(X,f) \mapsto (X,f,\id)$, and on morphisms we have $m \mapsto (m,m)$.

The partially observable systems $(X,f,q)$ where $q: X \to \tilde X$ is a morphism of $\cat{Sys}$ admit the projection $(X,f,q) \mapsto  (\tilde X, \tilde f)$. The archetypes of systems which admit these projections are the product systems.

\begin{remark}
We have an adjunction $U \vdash E$ where $U \circ E = \id_{\cat{Sys}}$. This exhibits $\cat{Sys}$ as a coreflective subcategory of $\cat{PSys}$.
\end{remark}

\section{Entropy and quotient entropy}

Given a sequence of positive numbers $\{ x_t \}$, we denote its exponential growth rate by
$$
\cat{GR}_t ( x_t ) \coloneqq \limsup_{t \to \infty} \frac{1}{t} \ln (x_t) .
$$
This growth rate is zero, if the sequence grows subexponentially, and infinite, if the sequence grows superexponentially.

The chief numerical invariant of a dynamical system $(X,f)$ is its topological entropy \cite{adler65}, the quantity
$$
h (X,f) \coloneqq \sup \left\{ \cat{GR}_n \left( \# \bigjoin_{i=0}^n f^{-i}\mathcal{U} \right) \Bigg| \text{$\mathcal{U}$ is open cover of $X$} \right\}
$$
where $\mathcal{A} \join \mathcal{B} \coloneqq \{ A \cap B\}_{A \in \mathcal{A}, B \in \mathcal{B}}$ and $\# \mathcal{C}$ denotes the minimal cardinality of a finite subcover of the open cover $\mathcal{C}$.

A well-known property of entropy is that it is functorial: If there is a diagram $(X,f) \to (Y,g)$, then $h(X,f) \geq h(Y,g)$. The assignment $\ent : \cat{Sys} \to \big( [0,\infty], \geq \big)$ that assigns $(X,f) \mapsto h(X,f)$ on objects is a functor.  (We consider the ordered set $( [0,\infty], \geq )$ as a thin category in the usual way.)

We want to quantify the loss of complexity due to partial observability. The quantity that we use for this purpose is \emph{quotient-topological entropy}. We define the quotient-topological entropy of the partially observable system $(X,f,q)$ as
\begin{align*}
&\tilde h (X,f,q) \\
&\coloneqq \sup \left\{ \cat{GR}_t \left( \# \bigjoin_{i=0}^n f^{-i}\mathcal{U} \right) \Bigg| \text{$\mathcal{U}$ is open cover of $X$ in the $q$-induced topology} \right\} \\
&= \sup \left\{ \cat{GR}_t \left( \# \bigjoin_{i=0}^n f^{-i}(q^{-1} \mathcal{U}) \right) \Bigg| \text{$\mathcal{U}$ is open cover of $X$} \right\} ,
\end{align*}
That the quotient-topological entropy is a well-defined number in $[0,\infty]$, follows from the same arguments used for topological entropy \cite{adler65}.

\begin{remark}
If $X$ is a metric space, the quotient-entropy may be computed by taking a limit along covers by quotient-metric balls of vanishing diameter. In fact, Bowen's construction \cite{bowen71} may be applied to the quotient metric.
\end{remark}

\begin{prop}
\label{functoriality}
The assignment $\qent: \cat{PSys} \to \big( [0,\infty] , \geq \big)$ where $\qent (X,f,q) = \tilde h (X,f,q)$ on objects is a functor.
\end{prop}
\begin{proof}
Let $(X,f,q_X) \to (Y,g,q_Y)$. The following diagram commutes.
$$
\begin{tikzcd}
X \arrow{r}{f} \arrow{d}{m} & X \arrow{r}{q_X} \arrow{d}{m} & \tilde X \arrow{d}{n} \\
Y \arrow{r}{g} & Y \arrow{r}{q_Y} & \tilde Y
\end{tikzcd}
$$
Open covers of $X$ in the topology induced by $q_X$ are equivalently open covers of $\tilde X$. Open covers of $Y$ in the topology induced by $q_Y$ are equivalently open covers of $\tilde Y$. Let $\mathcal{U}$ be such an open cover of $Y$. The assignment $\mathcal{U} \mapsto \{ m^{-1}(U) \}_{U \in \mathcal{U}}$ yields an open cover of $X$. This assignment respects the fibers of $n$.
\end{proof}

The following statements show that the observable complexity is at most equal to the complexity of the base system (\Cref{reduction}), and that, whenever two quotient maps are ordered with respect to their resolution, the respective quotient entropies behave monotonically (\Cref{mono}). They both follow from the observation that the cardinality of a minimal subcover is monotone on open covers.

\begin{prop}
\label{reduction}
Let $(X,f,q)$ be a partially observable system. Then $\tilde h (X,f,q_X) \leq h (X,f)$.
\end{prop}

\begin{prop}[Quotient entropy is antitone]
\label{mono}
Consider the system $(X,f)$. Let $q_0 : X \to Y_0$ and $q_1: X \to Y_1$ be quotient maps such that $\ker(q_0) \sqsubseteq \ker(q_1)$. Then $\tilde h (X,f, q_0) \geq \tilde h (X,f,q_1)$.
\end{prop}

If the quotient map is a morphism, the quotient-entropy of the partially observable system equals the entropy of the target.

\begin{prop}
\label{ifquotientismorphism}
Consider $(X,f,q_X)$ where $(X,f) \xrightarrow{q_X} (\tilde X , g)$. Then \newline $h(\tilde X, g) = \tilde h (X, f,q_X)$.
\end{prop}
\begin{proof}
Since the quotient map $q_X$ is a continuous map between compact Hausdorff spaces, $\ker (q_X) \simeq \tilde X$. It remains to note that the action by $g$ on $\tilde X$ is isomorphic to the action by $f$ on fibers of $q_X$.
\end{proof}

\begin{remark}
In the light of \Cref{ifquotientismorphism} and \Cref{mono}, one may use the quantity $h-\tilde h$ as a closure indicator. If the quotient map is an isomorphism, this quantity vanishes. Of course the reverse implication does not hold.
\end{remark}

The basic properties of entropy, additivity on products and multiplicativity under iteration, are shared by quotient entropy.

\begin{prop}
\label{properties}
The following two statements hold.
\begin{enumerate}[(i)]
\item Let $m \in \N$. The $m$-step partially observable system $(X, f^m , q_X)$ fulfills \newline $\tilde h (X, f^m, q_X ) = m \cdot \tilde h (X,f,q_X)$.
\item Let $(X, f, q_X)$ and $(Y,g, q_Y)$ be partially observable systems.\newline Then $\tilde h ( X \times Y, f \times g, q_X \times q_Y ) = \tilde h (X,f,q_X) + \tilde h (Y,g,q_Y)$.
\end{enumerate}
\end{prop}
\begin{proof}[Sketch of proof]
The proof makes use of basic properties \cite{adler65}. We sketch it for completeness.

\noindent $(i)$: For any open cover of $X$, we have
\begin{align*}
\cat{GR}_n \left( \# \bigjoin_{i=0}^n (f^m)^{-i} \mathcal{U} \right) &= \cat{GR}_n \left( \# \bigjoin_{i=0}^n f^{-mi} \mathcal{U} \right) = \limsup_{n \to \infty} \frac{1}{n} \ln \left( \# \bigjoin_{i=0}^n f^{-mi} \mathcal{U} \right) \\
&= m \cdot \limsup_{n \to \infty} \frac{1}{nm} \ln \left( \# \bigjoin_{i=0}^n f^{-im} \mathcal{U} \right) \\
&= m \cdot \limsup_{t \to \infty} \frac{1}{t} \ln \left( \# \bigjoin_{i=0}^t f^{-t} \mathcal{U} \right) \\
&= m \cdot \cat{GR}_t \left( \# \bigjoin_{i=0}^t f^{-t} \mathcal{U} \right) .
\end{align*}

\noindent $(ii)$: An open cover in the product topology of $X \times Y$ yields open covers of $X$ and $Y$ by projection. Also every open cover of $X \times Y$ arises as a product of open covers of $X$ and $Y$. The same reasoning applies to minimal subcovers. Let $\mathcal{U}$ be a minimal cover of $X$ and let $\mathcal{V}$ be a minimal cover of $Y$. Then $\mathcal{U} \times \mathcal{V}$ is a minimal cover of $X \times Y$ and $| \mathcal{U} \times \mathcal{V} | = |\mathcal{U}| \cdot |\mathcal{V}|$. We conclude that $\ln \big( | \mathcal{U} \times \mathcal{V} | \big) = \ln \big( |\mathcal{U}| \big) + \ln \big( |\mathcal{V}| \big)$.
\end{proof}

\section{Symbolic dynamics and the category of graphs}

\begin{deph}[The category $\cat{DiGraph}$]
The objects are finite square matrices with entries in $\{ 0,1 \}$. We identify matrices under the equivalence relation where $A \sim B$ if and only if there exists a permutation matrix $P$ such that $P^{-1} A P = B$. A morphism $A \xrightarrow{f} B$, where $A$ in $n \times n$ and $B$ is $m \times m$, is a surjection $f: \{1,...,n\} \to \{1,...,m\}$ such that $A_{ij} = 1$ implies $B_{f(i) f(j)} = 1$.
\end{deph}

Note that our graph morphisms do not include the proper graph embeddings. We recall the construction of symbolic dynamical systems. Given a finite set $\{1,...,l\}$ equipped with the discrete topology, we consider the shift space $\{1,...,l\}^\Z$ with the product topology. The map $\sigma: \{1,...,l\}^\Z \to \{1,...,l\}^\Z$ defined by $\sigma(s)_t = s_{t+1}$ is a homeomorphism. A subshift is a closed $\sigma$-invariant subset $S \subseteq \{1,...,l\}^\Z$. A special class of symbolic dynamical systems is given by the topological Markov chains.

\begin{deph}[Topological Markov chain]
Let $G$ be a finite directed graph with vertex set $\{1,...,l\}$. The subset $S_G \coloneqq \left\{ s \in \{1,...,l\}^\Z \big| \text{$G_{s_t,s_{t+1}} =1$ for all $t \in \Z$} \right\} \subseteq \{1,...,l\}^\Z$ is closed and $\sigma$-invariant. The system $(S_G, \sigma)$ is the topological Markov chain induced by $G$.
\end{deph}

It is well-known \cite{parry64}, that $h(S_G, \sigma) = \ln (\lambda_G^{\max})$, the natural logarithm of the spectral radius of the graph $G$. In the theory of dynamical systems, subshifts provide a link between continuous-topological approaches and discrete-algebraic ones. We recall the characterization of morphisms between subshifts.

\begin{thm}[Curtis, Lyndon, and Hedlund \cite{hedlund69}]
Let $S \subseteq \{1,...,s\}^\Z$ and $T \subseteq \{1,...,b\}^\Z$ be subshifts. Then the following two statements are equivalent.
\begin{enumerate}
\item The map $C: S \to T$ is continuous and $\sigma$-equivariant.
\item The map $C: S \to T$ is a sliding block code: There exists $m \in \N$ and a local defining map $c: \{1,...,a\}^{\{-m,...,m\}} \to \{1,...,b\}$ such that $C(w)_t = c \left( \{w_{t-m}, ..., w_{t+m} \} \right)$.
\end{enumerate}
\end{thm}

The Curtis-Lyndon-Hedlund theorem implies that the image of a topological Markov chain under a cellular automaton is a subshift, but not necessarily a topological Markov chain, as the following example shows.

\begin{eg}
Consider the subshift with the following Markov graph.
$$
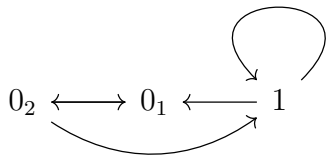
\begin{tikzcd}
0_2 \arrow{r} \arrow[bend right=35pt]{rr} & 0_1 \arrow{l} & 1 \arrow{l} \arrow[loop]{r}
\end{tikzcd}
$$
The $0$-memory cellular automaton $c: \{0_1,0_2,1\} \to \{0,1\}$ given by $0_1,0_2 \mapsto 0$ and $1 \mapsto 1$ maps this subshift of finite type to the even subshift
$$
\left\{ w \in \{0,1\}^\Z \big| \text{Consecutive $1$'s are separated by evenly many $0$'s} \right\} ,
$$
which is not a topological Markov chain.
\end{eg}

The following corollary characterizes $\sigma$-equivariance of $0$-memory cellular automata between topological Markov chains.
%Note that this is \emph{not yet} a characterization of the respective cellular automata but of embeddings of the respective actions.

\begin{cor}
\label{chainembedding}
Let $A$ be an $n \times n$ adjacency matrix and let $B$ be an $m \times m$ adjacency matrix. Let $c: \{1,...,n\} \to \{1,...,m\}$ be a surjection and let $C:  \{1,...,n\}^\Z \to \{1,...,m\}^\Z$ be the cellular automaton with local defining map $c: \{1,...,n\}^{\{0\}} \to \{1,...,m\}$. Then the graphs $A$ and $B$ are such that $A_{ij} = 1$ implies $B_{c(i) c(j)} = 1$.
\end{cor}
\begin{proof}
Suppose, aiming for a contradiction, that there exists a word $v \in S_A$ such that $A_{v_t v_{t+1}} = 1$ and $B_{c(v_t) c(v_{t+1})} = 0$ for some $t \in \Z$. This implies
$$
\left( \sigma \circ C (v) \right)_t = \sigma \left( c (v_t) \right) \neq c( v_{t+1}) = \left( C \circ \sigma (v) \right)_t ,
$$
contradicting $\sigma$-equivariance of $C$.
\end{proof}

Via topological Markov chains, the category $\cat{DiGraph}$ is functorially embedded into the category $\cat{Sys}$.

\begin{prop}
\label{equivalenceofcats}
We have a functorial embedding $S : \cat{DiGraph} \hookrightarrow \cat{Sys}$ given by $A \mapsto (S_A, \sigma)$ on objects and the graph morphism $f$ is assigned to $Sf: S_A \to S_B$ where $\left(Sf (w)\right)_t = f(w_t)$, a cellular automaton of memory length $0$. In particular, $\cat{DiGraph}$ is equivalent to the subcategory $\cat{TMChain_0}$ of topological Markov chains and $0$-memory sliding block codes.
\end{prop}
\begin{proof}[Sketch of proof]
Clearly $S$ is faithful. Suppose that $A \xrightarrow{f} B$ in $\cat{DiGraph}$. We want to show that $SA \xrightarrow{Sf} SB$, which is equivalent to the commutativity of the following diagram.
$$
\begin{tikzcd}
S_A \arrow{r}{\sigma} \arrow{d}{Sf} & S_A \arrow{d}{Sf} \\
S_B \arrow{r}{\sigma} & S_B
\end{tikzcd}
$$
Suppose that $A$ is $n \times n$ and that $B$ is $m \times m$. Let $w \in S_A$. Hence $A_{w_t w_{t+1}} = 1$ for all $t \in \Z$. Since $f: \{1,...,n\} \to \{1,...,m\}$ is a morphism, $A_{w_t w_{t+1}} = 1$ implies $B_{f(w_t) f(w_{t+1})} = 1$ and therefore $Sf (w) \in S_B$ by \Cref{chainembedding}.
\end{proof}

\begin{remark}
The categorical products in $\cat{DiGraph}$ correspond to the Kronecker-products of adjacency matrices. The embedding of graphs into dynamical systems respects finite products. We have $S_{A \otimes B} \, \simeq \, S_A \otimes S_B$.
\end{remark}

The existence of a graph morphism implies embeddability of the actions (\Cref{chainembedding}), not the existence of a morphism of systems. A necessary condition for the existence of a morphism can be obtained via the following lemma.

\begin{lemma}
Equivalences of categories preserve epimorphisms.
\end{lemma}

Let $A \xrightarrow{f} B$ in $\cat{DiGraph}$. Suppose that $g: B \to A$ is right-inverse to $f$, hence $ f \circ g = \id$. Then $Sg : S_B \to S_A$ is right-inverse to $Sf : S_A \to S_B$. In particular, $Sf$ is surjective. The existence of a section from the generating graph is a sufficient condition for topological Markov chains to be related by a morphism. The following proposition is immediate.

\begin{prop}
\label{condition}
Let $A$ be an $n \times n$ adjacency matrix and let $B$ be an $m \times m$ adjacency matrix. Consider a surjection $c: \{1,...,n\} \to \{1,...,m\}$. The induced map $C: S_A \to S_B$ is a morphism of dynamical systems if $c$ is a graph morphism which admits a right-inverse graph morphism.
\end{prop}

Note that, in general, it is NP-hard to verify whether a graph is the quotient of another.

\section{Quotient entropy and Markov partitions}\label{mpartitions}

There are several equivalent ways to define Markov partitions. We follow Adler \cite{adler98} and Gromov \cite{gromov16}.

\begin{deph}[Markov partition]
A Markov partition for the system $(X,f)$ is a finite collection $\{ U_i \}_{i=1}^n$, $n \geq 2$, of open sets $U_i \subset X$ such that the following three conditions hold.
\begin{enumerate}[(i)]
\item $U_i \cap U_j = \varnothing$ whenever $i \neq j$.
\item $\bigcup_{i = 1}^n U_i$ is dense in $X$.
\item  For all $n,m \in \N$, if simultaneously $f^n ( U_i ) \cap U_j \neq \varnothing$ and $f^m ( U_j ) \cap U_k \neq \varnothing$ hold, then $f^{n+m} ( U_i ) \cap U_k \neq \varnothing$.
\end{enumerate}
\end{deph}

The existence of a Markov partition for a dynamical system is a demanding property. Bowen \cite{bowen70} showed that a Markov partition exists for any system satisfying Axiom A \cite{smale67}. Explicit constructions go back to Sinai \cite{sinai68}. The usefulness of Markov partitions originates in the following construction.

\begin{deph}[Hadamard \cite{hadamard98}]
We assign to a dynamical system $(X,f)$ with Markov partition $\{ U_i \}_{i=1}^n$ the $n \times n$ adjacency matrix $A$ where
$$
A_{ij} = \begin{cases}
\text{$1$ if $f \left( U_i \right) \cap U_j \neq \varnothing$} \\
\text{$0$ otherwise.}
\end{cases}
$$
\end{deph}

The Hadamard construction and the embedding $S: \cat{DiGraph} \hookrightarrow \cat{Sys}$ have some properties of a retraction. Let $(X,f)$ be a dynamical system with Markov partition $\{ U_i \}_{i=1}^n$. Denote by $\phi_i (x)$ the index of the partition element that contains the image of $x \in X$ under $f^i$ up to closure, hence $f^i (x) \in \overline{U}_{\phi_i (x)}$. The assignment $\Phi: X \to [n]^\N$ is defined as $\Phi (x) = \left\{ \phi_i (x) \right\}_{i \in \N}$. This assignment may be one-to-many on some points: We suppose that some choice is made. The pseudo-inverse assignment $\Phi^{-1} : [m]^\N \to 2^X$ is obtained in the following manner. For any $w \in W_A$, define $R_0 (w) = \left\{ x \in X_{w_0} \right\}$, and, iteratively for $t \in \N$,
\begin{equation*}
R_t (w) = \left\{ x \in R_{t-1} (w) : f(x) \in X_{w_t} \right\} .
\end{equation*}
Then set $\Phi^{-1} (w) \coloneqq \bigcap_{t=0}^\infty R_t (w) \subset X$. Often, for example if $f$ is a minimal homeomorphism, we have $| \Phi^{-1} (w) | = 1$.

We are interested in systems that are \emph{finitely presented}, in the sense that they admit a Markov partition such that the Hadamard construction followed by the embedding $S$ yields a homeomorphism. In that case the Hadamard construction yields an isomorphism between the original system and a topological Markov chain. These structural relationships, as well as others appearing in this work, are illustrated in the following diagram.

$$
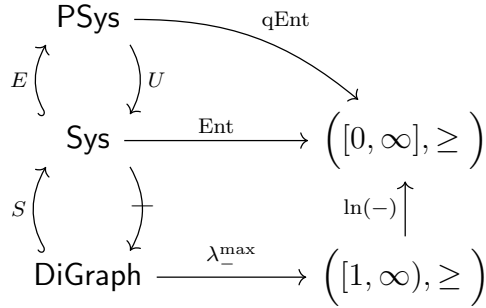
\label{diagram}
\begin{tikzcd}[column sep = huge]
\cat{PSys} \arrow[shift left=10pt, bend left]{d}{U} \arrow[bend left=20pt]{rd}{\qent} & {} \\
\cat{Sys}  \arrow[shift left=10pt, bend left]{d}[rotate=90,xshift=-3pt]{|} \arrow[hookrightarrow,shift left=10pt, bend left]{u}{E} \arrow{r}{\ent} & \Big( [0,\infty] , \geq \Big) \\
%\cat{Sys}  \arrow["/"{anchor=center,sloped},shift left=10pt, bend left]{d} \arrow[hookrightarrow,shift left=10pt, bend left]{u}{E} \arrow{r}{\ent} & \Big( [0,\infty] , \geq \Big) \\
\cat{DiGraph} \arrow[hookrightarrow, shift left=10pt, bend left]{u}{S} \arrow{r}{\lambda_{-}^{\max}} &  \Big( [1,\infty) , \geq \Big) \arrow{u}{\ln(-)}
\end{tikzcd}
$$

\begin{eg}[Expanding circle maps]
\label{motivation}
Consider $[0,1]_\sim$, the unit interval with identified endpoints. The family of maps $E_n : [0,1]_\sim \to [0,1]_\sim$ where $E_n (x) = n x \mod 1$ is finitely presented via the Markov partitions
\begin{equation*}
\left\{ \left( \dfrac{i-1}{n} , \dfrac{i}{n} \right) \right\}_{i=1}^{n} .
\end{equation*}
These partitions correspond to the complete graph on $n$ vertices. (See, for example, \cite[Paragraph 1.3]{brin02}.) Consider the map $E_4$ with the Markov partition into quadrants. We consider the quotient by vertical projection $x \mapsto \cos ( 2 \pi x )$ and the map $\{1,4\} \mapsto 1$, $\{2,3\} \mapsto 2$.
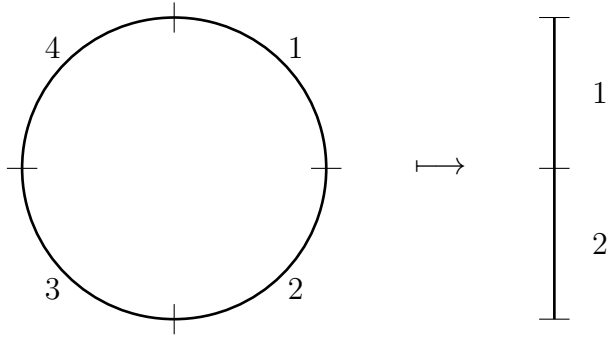
\begin{equation*}
\begin{tikzpicture}
\draw[line width=1pt] (0,0) circle (2);
\draw (0,2.2) -- (0,1.8);
\draw (0,-2.2) -- (0,-1.8);
\draw (2.2,0) -- (1.8,0);
\draw (-2.2,0) -- (-1.8,0);
\draw (1.6,1.6) node {$1$};
\draw (1.6,-1.6) node {$2$};
\draw (-1.6,-1.6) node {$3$};
\draw (-1.6,1.6) node {$4$};
\draw (3.5,0) node {$\longmapsto$};
\draw[line width=1pt] (5,-2) -- (5,2);
\draw (4.8,2) -- (5.2,2);
\draw (4.8,-2) -- (5.2,-2);
\draw (4.8,0) -- (5.2,0);
\draw (5.6,1) node {$1$};
\draw (5.6,-1) node {$2$};
\end{tikzpicture}
\end{equation*}
The Markov graph is the complete graph on four vertices and hence $h(E_4) = \ln (4)$. Clearly no projection of the circle onto the interval is a morphism of the system. We are tempted to say that the quotient map reduces the symbolic dynamics to the full shift on two symbols and that this should imply $\tilde h (E_4) = \ln (2)$.
\end{eg}

\noindent
We proceed by formalizing the above example and by proving its correctness.

\begin{deph}[Compatible topological partition]
Let $(X,f,q)$ be a partially observable system with Markov partition $\mathcal{X} = \{ U_i \}_{i=1}^n$ for $(X,f)$. Consider a surjection $s: \{1,...,n\} \to \{1,...,m\}$ such that $\mathcal{Y} \coloneqq \left\{ q(U_{s(i)}) \right\}_{s(i) \in [m]} \coloneqq \left\{ V_j \right\}_{j \in [m]}$ fulfills $ \mathcal{X} \sqsubseteq q^{-1} \mathcal{Y} $. Then $\mathcal{Y}$ is a compatible topological partition.
\end{deph}

\begin{cor}[Corollary of the definition]
We have a graph morphism $c: A \to B$ where $A$ is the Markov graph of $(X,f)$ given by
$$
A_{ij} = \begin{cases}
\text{$1$ if $f(U_i) \cap U_j \neq \varnothing$}\\
\text{$0$ otherwise,}
\end{cases}
$$
and $B$ is defined by
$$
B_{ij} = \begin{cases}
\text{$1$ if $f \left( q^{-1} V_i \right) \cap q^{-1} V_j \neq \varnothing $} \\
\text{$0$ otherwise.}
\end{cases}
$$
\end{cor}
\begin{proof}
Suppose $A_{ij} = 1$, which is equivalent to $f(U_i) \cap U_j \neq \varnothing$. We have $U_i \subseteq q^{-1} V_{c(i)}$, which is equivalent to $f(U_i) \subseteq f( q^{-1} V_{c(i)})$, and $U_j \subseteq q^{-1} V_{c(j)}$. We conclude $f \left( q^{-1} V_{c(i)} \right) \cap q^{-1} V_{c(j)} \neq \varnothing$, which is equivalent to $B_{c(i) c(j)} = 1$. Hence $c$ is a graph morphism.
\end{proof}

\begin{prop}
\label{fancydiagram}
Let $(X,f)$ be a finitely presented dynamical system with Markov partition $\mathcal{X} = \{ U_i \}_{i=1}^n$. Consider the partially observable system $(X,f,q)$ with quotient map $q: X \to Y$. Suppose that the compatible topological partition $\mathcal{Y} = \{ V_i \}_{i=1}^m$ of $Y$ induces the graph morphism $c: A \to B$. Suppose that $c$ admits a right-inverse graph morphism. Then $(S_B , \sigma)$ is a symbolic dynamical system such that the diagram
\begin{equation*}
\begin{tikzcd}
X \arrow{r}{f} \arrow{d}{\Phi} & X \arrow{d}{\Phi} \arrow{r}{q} & Y \arrow{dd}{h} \\
S_A \arrow{d}{C} \arrow{r}{\sigma} & S_A \arrow{dr}{C} & {} \\
S_B \arrow{rr}{\sigma} & {} & S_B
\end{tikzcd}
\end{equation*}
commutes where $h: Y \to S_B$ is a homeomorphism and $C: \{1,...,n\}^\Z \to \{1,...,m\}^\Z$ is the $0$-memory cellular automaton with local defining map $c$. In particular, this implies $\tilde h (X,f,q) = h(S_B, \sigma)$.
\end{prop}
\begin{proof}
By \Cref{condition}, the map $S$ is a $0$-memory cellular automaton between subshifts of finite types, since $s$ admits a right-inverse graph morphism. Since $S: S_A \to S_B$ is a morphism, the respective square in the above diagram commutes.

\noindent
We now construct the map $h$. By hypothesis, a trajectory through the cells of $\mathcal{X}$ uniquely identifies a point in $X$. A trajectory through $\mathcal{Y}$ is equivalent to a trajectory through its pullback onto $X$. Such a trajectory only identifies subsets of $X$ and not points. We extend $\Phi^{-1} : 2^{W_A} \to 2^X$ by direct images. The restriction $\Phi^{-1} |_{\ker(S)} : \ker(S) \to \ker (q)$ is a continuous bijection of compact Hausdorff spaces and therefore a homeomorphism. By \Cref{functoriality} we know that $\tilde h (X,f,q) = \tilde h (S_A, \sigma, S)$. But we have also shown that $(S_A, \sigma) \xrightarrow{S} (S_B, \sigma)$. By \Cref{ifquotientismorphism} we have $\tilde h (S_A, \sigma, S) = h (S_B, \sigma)$.
\end{proof}

Note that the above proposition applies to \Cref{motivation}.

\begin{eg}
Consider the circle map $E_3$, see \Cref{motivation}, and the corresponding Markov partition.
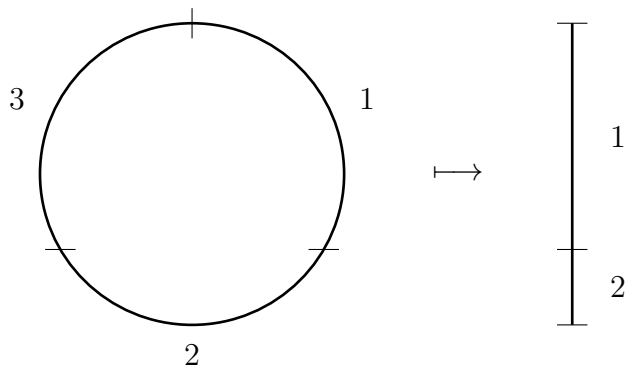
\begin{equation*}
\begin{tikzpicture}
\draw[line width=1pt] (0,0) circle (2);
\draw (0,2.2) -- (0,1.8);
\draw (-1.532, -1) -- (-1.932, -1);
\draw (1.532, -1) -- (1.932, -1);
\draw (2.3,1) node {$1$};
\draw (-2.3,1) node {$3$};
\draw (0,-2.4) node {$2$};
\draw (3.5,0) node {$\longmapsto$};
\draw[line width=1pt] (5,-2) -- (5,2);
\draw (4.8,2) -- (5.2,2);
\draw (4.8,-2) -- (5.2,-2);
\draw (4.8, -1) -- (5.2, -1);
\draw (5.6,0.5) node {$1$};
\draw (5.6,-1.5) node {$2$};
\end{tikzpicture}
\end{equation*}
The dynamic of the compatible Markov partition of the interval is the full shift on two symbols. Hence $\tilde h (E_3) = \ln (2)$. (Combining with \Cref{motivation}, we have obtained a case where $\tilde h (X,f,q) = \tilde h (X,g,q)$ for $f \neq g$.)
\end{eg}

\begin{eg}
Consider the circle map $E_3$, see \Cref{motivation}, and the horizontal projection $x \mapsto \sin ( 2 \pi x)$. The images of the cells of the Markov partition overlap in the unit interval in a way that makes a compatible selection impossible.

$$
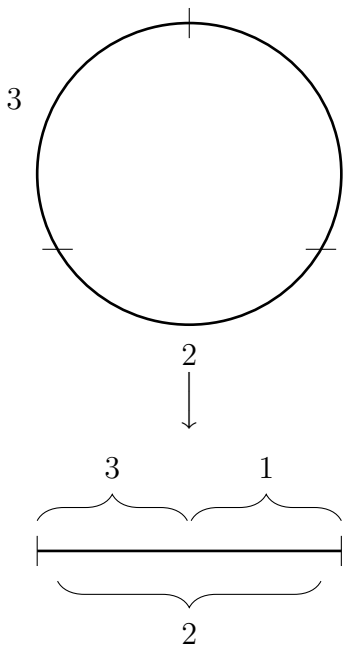
\begin{tikzpicture}
\draw[line width=1pt] (0,0) circle (2);
\draw (0,2.2) -- (0,1.8);
\draw (-1.532, -1) -- (-1.932, -1);
\draw (1.532, -1) -- (1.932, -1);
\draw (2.3,1) node {$1$};
\draw (-2.3,1) node {$3$};
\draw (0,-2.4) node {$2$};
\draw (0,-3) node {$\Big\downarrow$};
\draw[line width=1pt] (-2,-5) -- (2,-5);
\draw (-2,-4.8) -- (-2,-5.2);
\draw (2,-4.8) -- (2,-5.2);
\draw [decorate,decoration={brace,amplitude=10pt}]  (-2,-4.6)--(-0.025,-4.6) node [black,midway,yshift=20pt] {$3$};
\draw [decorate,decoration={brace,amplitude=10pt}]  (0.025,-4.6)--(2,-4.6) node [black,midway, yshift=20pt] {$1$};
\draw [decorate,decoration={brace,amplitude=10pt}]  (1.732,-5.4)--(-1.732,-5.4) node [black,midway, yshift=-20pt] {$2$};
\end{tikzpicture}
$$
\end{eg}

We close by discussing a class of examples of our scheme for calculating quotient entropy. Every horseshoe map of the interval is finitely presented \cite[Paragraph 15.1]{katok97}. Under an appropriate restriction of admissible quotient maps, there always is a quotient chain.

\begin{deph}[Horseshoe map]\label{horseshoe}
A continuous map $f: [0,1] \to [0,1]$ is a horseshoe map if the following three conditions hold.
\begin{enumerate}[(i)]
\item There exists $n \in \N$, $n \geq 2$, such that the set $\left\{ 0, 1/n, 2/n , ..., (n-1)/n , 1 \right\}$ is $f$-invariant.
\item The restriction $f|_{[ (i-1)/n , i/n]}$ is nonconstant and affine for every $i \in [n]$.
\item There are no sinks in the Markov graph $F \in \mat_{n \times n} \left( \{0,1\} \right)$ given by
$$
F_{ij} = \begin{cases}
\text{$1$ if $f \left( \left( \frac{i-1}{n} , \frac{i}{n} \right) \right) \cap \left( \frac{j-1}{n} , \frac{j}{n} \right) \neq \varnothing$}\\
\text{$0$ otherwise.}
\end{cases}
$$
\end{enumerate}
\end{deph}

A Markov partition for a horseshoe map is given by $\left\{ \left( \frac{i-1}{n} , \frac{i}{n} \right) \right\}_{i=1}^n$. If we only consider certain quotient maps of the interval, we obtain a class of partially observable systems whose properties are straightforward.

\begin{deph}[Good quotient map for horseshoes]
Let $\left( [0, 1] , f \right)$ be a dynamical system where $f$ has horseshoe structure $\left\{ 0, 1/n, 2/n , ..., (n-1)/n , 1 \right\}$. The quotient map $q:[0,1] \to [0,1]$ is good if the following three conditions hold.
\begin{enumerate}[(i)]
\item $q \left( \left\{ 0, 1/n, 2/n , ..., (n-1)/n , 1 \right\} \right) \subseteq \left\{ 0, 1/n, 2/n , ..., (n-1)/n , 1 \right\}$
\item The restriction $q|_{[(i-1)/n , i/n]}$ is affine for every $i \in [n]$.
\item $q$ is monotone.
\end{enumerate}
\end{deph}

A nontrivial good quotient map must be constant on at least one interval. Monotonicity and surjectivity imply that $q(0) = 0$ and $q(1) = 1$.

\begin{prop}
Consider $\left( [0,1] , f , q \right)$ where $f : [0,1] \to [0,1]$ is a horseshoe map and $q: [0,1] \to [0,1]$ is a good quotient. Then there is a compatible selection corresponding to a right-invertible graph morphism.
\end{prop}
\begin{proof}
Denote by $\mathcal{X} = \{ U_i \}_{i=1}^n$ the Markov partition of $[0,1]$ corresponding to $f$. Consider the image of the partition under the quotient map $q ( \mathcal{X} ) \coloneqq \left\{ q(U_i) \right\}_{i=1}^n$. Extract the subset $\mathcal{Y} \coloneqq \left\{ V \in q ( \mathcal{X} ) : \text{$V$ is open} \right\}$. Clearly $ | \mathcal{Y} | \leq | \mathcal{X} |$. Note that $\mathcal{Y}$ consists of mutually disjoint open sets whose union is dense in $[0,1]$. We have $\mathcal{X} \sqsubseteq q^{-1} \mathcal{Y}$. Since $q$ is monotone, we may enumerate the subsets $C_j = q^{-1} ( V_j )$ increasingly on $j \in [m]$ where $m \leq n$. Define $c: \{1, ..., n\} \to \{1,...,m\}$ by setting $c(i) = \min_{j \in [m]} \left\{ j: V_i \subseteq C_j \right\}$. For all $C_j$ there exists a unique $V_i$ such that $V_i \subseteq \left( C_j \setminus \cup_{k \neq j} C_k \right)$. Let $g: \{1,...,m\} \to \{1,...,n\}$ be this assignment. We have obtained $\mathcal{Y} = \{ V_j \}_{j=1}^m$ and $c: \{1,...,n\} \to \{1,...,m\}$ such that \Cref{fancydiagram} applies.
\end{proof}

\begin{eg}
\label{standard}
Consider the following horseshoe map $f$ and good quotient $q$.
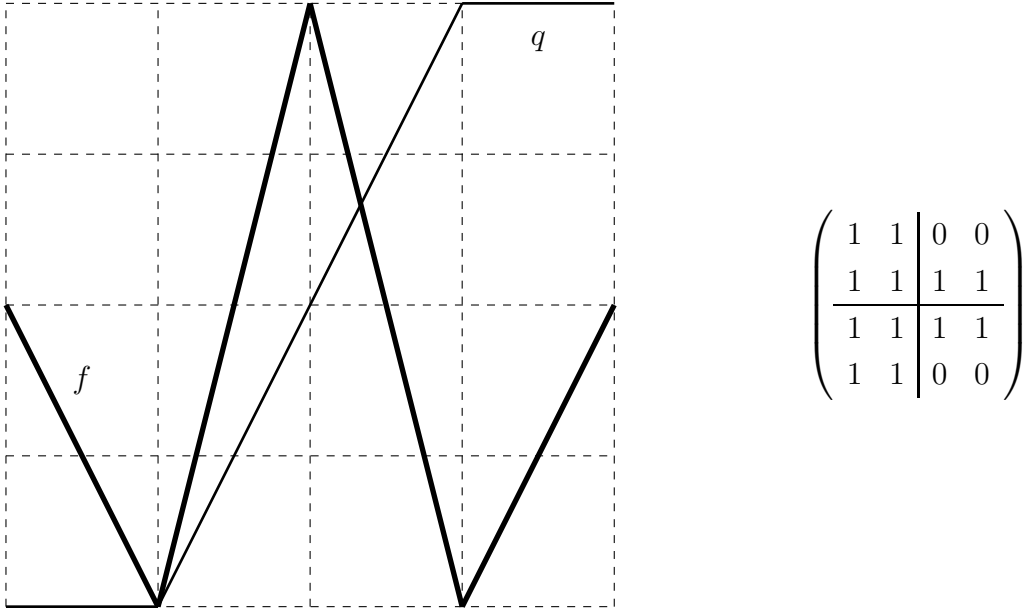
\begin{equation*}
\begin{tikzpicture}
\draw[dashed] (0,0) -- (8,0);
\draw[dashed] (0,2) -- (8,2);
\draw[dashed] (0,4) -- (8,4);
\draw[dashed] (0,6) -- (8,6);
\draw[dashed] (0,8) -- (8,8);

\draw[dashed] (0,0) -- (0,8);
\draw[dashed] (2,0) -- (2,8);
\draw[dashed] (4,0) -- (4,8);
\draw[dashed] (6,0) -- (6,8);
\draw[dashed] (8,0) -- (8,8);

\draw[line width = 2pt] (0,4) -- (2,0);
\draw[line width = 2pt] (2,0) -- (4,8);
\draw[line width = 2pt] (4,8) -- (6,0);
\draw[line width = 2pt] (6,0) -- (8,4);

\draw[line width = 1pt] (0,0) -- (2,0);
\draw[line width = 1pt] (2,0) -- (6,8);
\draw[line width = 1pt] (6,8) -- (8,8);

\draw (7,7.5) node {$q$};
\draw (1,3) node {$f$};

\draw (12,4) node {$
\left(
\begin{array}{c c | c c}
1 & 1 & 0 & 0 \\
1 & 1 & 1 & 1 \\
\hline
1 & 1 & 1 & 1 \\
1 & 1 & 0 & 0
\end{array}
\right)$};
\end{tikzpicture}
\end{equation*}
We have $h([0,1],f) = \ln (3)$. The map given by $1,2 \mapsto 1$, $3,4 \mapsto 2$ is a graph morphism with right-inverse $1 \mapsto 2$, $2 \mapsto 3$. We have $\tilde h ( [0,1], f, q) = \ln(2)$.
\end{eg}

\printbibliography

\end{document}